\documentclass{elsarticle}
 \usepackage{graphicx,amssymb,amstext,amsmath}
 \usepackage{bigstrut}
 \usepackage{enumerate}
 \usepackage{amsthm}
 \usepackage[english]{babel}
 \usepackage[all]{xy}
 \usepackage[vcentermath]{youngtab}
\usepackage{longtable}

 \newtheorem{theorem}{Theorem}[section]
 \newtheorem{thm}[theorem]{Theorem}
 
 \newtheorem{lem}[theorem]{Lemma}
 \newtheorem{cor}[theorem]{Corollary}
 \newtheorem{prop}[theorem]{Proposition}

\DeclareMathOperator\ID{\mathbf{1}}
\DeclareMathOperator\id{id}

\DeclareMathOperator\proj{proj}

\DeclareMathOperator\alt{Alt}
\DeclareMathOperator\sym{Sym}
\DeclareMathOperator\PGL{PGL}

\DeclareMathOperator\irr{Irr}

\DeclareMathOperator\Spec{Spec}

\DeclareMathOperator\dd{\mathcal{D}}

\begin{document}
\title{The Erd\H{o}s-Ko-Rado property for some permutation groups}

 \author{Bahman Ahmadi}
 \ead{ahmadi2b@uregina.ca}

 \author{Karen Meagher \corref{cor1}\fnref{fn1}}
 \ead{karen.meagher@uregina.ca}
 \cortext[cor1]{Corresponding author}
 \fntext[fn1]{Research supported by NSERC.}

 \address{Department of Mathematics and Statistics,\\
 University of Regina, 3737 Wascana Parkway, S4S 0A4 Regina SK, Canada}

 \begin{abstract}
   A subset in a group $G \leq \sym(n)$ is \textsl{intersecting} if
   for any pair of permutations $\pi,\sigma$ in the subset there is an $i \in
   \{1,2,\dots,n\}$ such that $\pi(i) = \sigma(i)$. If the stabilizer
   of a point is the largest intersecting set in a group, we say that
   the group has the \textsl{Erd\H{o}s-Ko-Rado (EKR)
     property}. Moreover, the group has the \textsl{strict EKR}
   property if every intersecting set of maximum size in the group is
   either the stabilizer of a point or the coset of the stabilizer of
   a point.  In this paper we look at several families of permutation
   groups and determine if the groups have either the EKR property or
   the strict EKR property. First, we prove that all cyclic
   groups have the strict EKR property. Next we show that  all dihedral and Frobenius
   groups have the EKR property and we characterize which ones have
   the strict EKR property. Further, we show that if all the groups in an
   external direct sum or an internal direct sum have the EKR (or
   strict EKR) property, then the product does as well. Finally, we
   show that the wreath product of two groups with EKR property also
   has the EKR property.
 \end{abstract}


\maketitle

\section{Introduction}\label{introduction}

The Erd\H{o}s-Ko-Rado (EKR) theorem~\cite{TheOriginalEKR} describes the
largest collection of subsets of size $k$ from the set
$\{1,2,\dots,n\}$ such that any two of the subsets contain a common
element. If $n \geq 2k$ then the largest collection has size
$\binom{n-1}{k-1}$, moreover, if $n>2k$, the only collections of
this size are the collections of all subsets that contain a fixed
element from $\{1,2,\dots,n\}$. This result can be generalized to many
objects other than sets, such as integer sequences~\cite{MR1722210},
vector spaces over a finite field~\cite{MR0382015}, matchings and
partitions~\cite{MR2156694}. In this paper we will consider versions
of the EKR theorem for permutation groups.

Let $\sym(n)$ denote the symmetric group and $G\leq \sym(n)$ be a
permutation group with the natural action on the set
$\{1,2,\dots,n\}$.  Two permutations $\pi,\sigma\in G$ are said to
\textsl{intersect} if $\pi\sigma^{-1}$ has a fixed point in
$\{1,2,\dots,n\}$.  If $\pi$ and $\sigma$ intersect, then $\pi(i) =
\sigma(i)$ for some $i \in [n] = \{1,2,\dots,n\}$ and we say that
$\pi$ and $\sigma$ \textsl{agree} at $i$. Further, if $\sigma$ and
$\pi$ do not intersect, then $\pi\sigma^{-1}$ is a derangement. A
subset $S\subseteq G$ is called \textsl{intersecting} if any pair of
its elements intersect. Clearly, the stabilizer of a point is an
intersecting set in $G$, as is any coset of the stabilizer of a point.

We say the group $G$ has the \textsl{EKR property} if the size of any
intersecting subset of $G$ is bounded above by the size of the largest
point-stabilizer in $G$. Further, $G$ is said to have the
\textsl{strict EKR property} if the only intersecting subsets of maximum size
in $G$ are the cosets of the point-stabilizers. It is clear from the
definition that if a group has the strict EKR property then it will
have the EKR property.

In 1977 Frankl and Deza \cite{Frankl1977352} proved that $\sym(n)$ has
the EKR property and conjectured that it had the strict EKR
property. This conjecture caught the attention of several researchers,
indeed, it was proved using vastly different methods in each of 
\cite{MR2009400, Karen,MR2061391} and \cite{MR2419214}. We state this
result using our terminology.

\begin{thm}\label{EKR_for_sym}
For all integers $n\geq 2$, the group $\sym(n)$ has the strict EKR property. \qed
\end{thm}

Researchers have also worked on finding other subgroups of $\sym(n)$
that have the strict EKR property. For example in~\cite{KuW07} it is
shown that $\alt(n)$ has the strict EKR property, and it was also
determined exactly which Young subgroups have the strict EKR
property. In~\cite{MR2419214} it is shown that some Coxeter groups
have the EKR property. Further, the projective general linear groups
$\PGL(2,q)$ have the strict EKR property~\cite{MeagherS11}, while the
groups $\PGL(3,q)$ do not~\cite{MeagherS13}.

We point out that the group action is essential for the concepts of
the EKR and the strict EKR properties. A group can have the (strict) EKR
property under one action while it fails to have this property under
another action (in Section~\ref{non-examples} we give an example of
such a group). This is the reason that in this paper, we always
consider the ``permutation groups'' (i.e. the subgroups of $\sym(n)$
with their natural action on $\{1,\dots,n\}$) rather than general
groups. This is in contrast to~\cite{bardestani} where they define a
group to have the EKR property if the group has the EKR property for
every group action.

In this paper we prove that any group with a sharply transitive set
has the EKR property and that all cyclic groups have the strict EKR
property. We also show that all dihedral groups have the strict EKR
property, all Frobenius groups have the EKR property and we
characterize which Frobenius groups have the strict EKR property. We
show that if all the groups in either an external or an internal
direct product have the EKR (or strict EKR) property, then the product
does as well; further, we show that if the external product of a set
of groups has the EKR (or strict EKR) property then each of the group
in the product does as well. The wreath product of two groups with EKR
property also has the EKR property.  We conclude with a section about
some groups that do not have the EKR property.

\section{The EKR property for cyclic groups}

A common technique in proving EKR theorems is to define a graph that
has the property that the independent sets in the graph are exactly
the intersecting sets. This transforms the problem to determining the
size and the structure of the independent sets in this graph. For the
EKR theorem on sets, the graph used is the Kneser graph
(see~\cite[Section 7.8]{MR1829620} for more details). For a group $G$
we define the \textsl{derangement graph}; this graph has the elements
of $G$ as its vertices and two vertices are adjacent if and only if
they do not intersect. We will denote the derangement graph of a group
by $\Gamma_G$.  In particular, $\sigma, \pi \in G$ are adjacent in
$\Gamma_G$ if and only if $\sigma \pi^{-1}$ is a derangement. The
derangement graph is the Cayley graph on $G$ where the set of all
derangements of $G$ is the connection set for the graph. Since the set
of derangements is a union of conjugacy class $\Gamma_G$ is a
\textsl{normal} Cayley graph and it is vertex transitive (the group
$G$ is a transitive subgroup of the automorphism group for
$\Gamma_G$).

The group $G$ has the EKR property if and only if the size of a
maximum independent set in $\Gamma_G$ is equal to the size of the
largest stabilizer of a point in $G$. The group $G$ has the strict EKR
property if and only if the cosets of the largest stabilizer of a
point are all the only independent sets of maximum size.

Following the standard notation, we will denote the size of the
largest clique in a graph $X$ by $\omega(X)$, and $\alpha(X)$ will
represent the size of the largest independent set in $X$. The next
result is known as the \textsl{clique-coclique bound} and we refer the
reader to \cite{MR2009400} for a proof.

\begin{lem}\label{clique_coclique_bound}
  IF $X$ is a vertex-transitive graph, then $\omega(X) \alpha(X) \leq
  |V(X)|$.  Moreover, if equality holds, then every maximum
  independent set and every maximum clique intersect.\qed
\end{lem}

For any permutation group $G \leq \sym(n)$, a clique in $\Gamma_G$ is
a set of permutations in which no two permutations agree on a
point. Thus a sharply-transitive subset of $G$ is a clique of size $n$ in $\Gamma_G$.

\begin{cor}
If a group has a sharply-transitive set, then the group has the EKR property.
\end{cor}
\begin{proof}
  Assume that a group $G \leq \sym(n)$ has a sharply-transitive
  set. Then $G$ is a transitive group and the size of the largest
  stabilizer of a point is $\frac{|V(G)|}{n}$. The sharply transitive
  subset of $G$ is a clique of size $n$ in $\Gamma_G$. So by the
  clique-coclique bound $\alpha(\Gamma_G) \leq \frac{|G|}{n}$.
\end{proof}

A simple case of a sharply-transitive set is the subgroup generated
by an $n$-cycle.

\begin{cor}
  Any permutation group of degree $n$ that contains an $n$-cycle
  has the EKR property.\qed
\end{cor}

If $G$ is a sharply-transitive group, then $\Gamma_G$ is a complete
graph; in this case $G$ trivially has the strict EKR theorem. For
example, any subgroup of $\sym(n)$ that is generated by an $n$-cycle
has the strict EKR property. Next we show that every cyclic group has
the strict EKR property.

\begin{thm}\label{cyclic_groups}
  For any permutation $\sigma\in \sym(n)$, the cyclic group $G$
  generated by $\sigma$ has the strict EKR property.
\end{thm}
\begin{proof}
  Let $\sigma=\sigma_1\sigma_2\cdots\sigma_k$, where $\sigma_i$ are
  disjoint cyclic permutations. Assume that $\sigma_i$ has order
  $r_i$, and that $1\leq r_1\leq\cdots\leq r_k$. Note that 
\[
C=\{\sigma^0, \sigma^1, \sigma^2,\ldots,\sigma^{r_1 - 1}\}
\]
induces a clique of size $r_1$ in the graph $\Gamma_G$. Then by
Lemma~\ref{clique_coclique_bound}
\[
\alpha(\Gamma_G)\leq \frac{|G|}{r_1}.
\]
Note, in addition, that if $\sigma_1=(a_1,\ldots,a_{r})$, then the
stabilizer of $a_1$ in $G$ is
\[
G_{a_1}=\{\sigma^{r_1}, \sigma^{2r_1},\ldots, \sigma^{|G|} \};
\]
therefore
\[
\alpha(\Gamma_G)=|G_{a_1}|=\frac{|G|}{r_1}.
\]
It is clear that this is the largest size of a point-stabilizer in
$G$. This proves that $G$ has the EKR property.

To show $G$ has the strict EKR property, note that the clique $C$ and
the independent set $G_{a_1}$ together show that the clique-coclique
bound (Lemma~\ref{clique_coclique_bound}) for $\Gamma_G$ holds with
equality. Hence any maximum independent set must intersect with any
maximum clique in $\Gamma_G$. Let $S$ be any maximum independent set
and without loss of generality assume the identity element, $\id$, is
in $S$. We show that $S$ is the stabilizer of a point.

For any $i\geq 0$, set
\[
C_i=\{\sigma^{i},\sigma^{i+1},\dots,\sigma^{i+r_1-1}\}.
\]
Note that $C_i$ are cliques in $\Gamma_G$ of maximum size and that
$C_0=C$. Furthermore, for any $0\leq t\leq |G|/r_1$, we have
$C_{tr_1}\backslash C_{tr_1+1}=\{\sigma^{tr_1}\}$ and
$C_{tr_1+1}\backslash C_{tr_1}=\{\sigma^{(t+1)r_1}\}$. Since for any
$0\leq t\leq |G|/r_1$, the independent set $S$ intersects with each of
the cliques $C_{tr_1+1} $ and $ C_{tr_1}$ in exactly one point, and
since $\sigma^0\in S$, we conclude that $C_{1} \cap S = \{
\sigma^{r_1}\}$.  Continue like this for all $t \in \{0, \dots,
|g|/r_1\}$ shows that $\sigma^0, \sigma^{r_1}, \ldots,
\sigma^{|G|-r_1}\in S$; that is, $S=G_{a_1}$.
\end{proof}

If $G$ is a cyclic group generated by $\sigma\in \sym(n)$, and $r_1$
is the length of the shortest cycle in $\sigma$, then the largest
intersecting set is actually an $r_1$-intersecting set; meaning that
any two permutations in the set agree on $r_1$ elements of $\{1,2,\dots,n\}$. This
is quite different from the case for subsets; in fact, it is not hard
to show that any maximal collection of intersecting subsets is not
$2$-intersecting.

\section{The EKR property for dihedral and Frobenius groups}

Recall that for any $n\geq 3$, the \textsl{dihedral group} of degree
$n$, denoted by $D_n\leq \sym(n)$, is the group of symmetries of a regular
$n$-gon, including both rotations and reflections.

\begin{prop}\label{dihedral}
All the dihedral groups have the strict EKR property.
\end{prop}
\begin{proof}
  Assume $D_n$ is generated by the permutations $\sigma$, the
  rotation, and $\pi$, the reflection through the antipodal points of
  the $n$-gon. Then $\sigma$ is of order $n$, and $\pi$ is of order
  2. Since $\{\id,\pi\}$ is an intersecting set, we have
  $\alpha(\Gamma_{D_n})\geq 2$. To prove the proposition, we show that
  any maximum independent set in $\Gamma_{D_n}$ is a coset of a
  point-stabilizer. Assume $S$ is a maximum independent set in
  $\Gamma_{D_n}$ and, without loss of generality, assume $\id\in
  S$. Clearly, $\sigma^i\notin S$, for any $1\leq i<n$. If
  $\sigma^i\pi, \sigma^j\pi\in S$, for some $1\leq j<i<n$, then their
  division,
\[
\sigma^i\pi(\sigma^j\pi)^{-1}=\sigma^{i-j}
\]
must have a fixed-point, which is a contradiction. Therefore,
$\alpha(\Gamma_{D_n})=2$ and $S=\{e,\sigma^i\pi\}$, for some $1\leq
i<n$. Note, finally, that since no pair $\sigma^i\pi, \sigma^j\pi$
have any common fixed-point, $S$ is indeed the stabilizer of any of
the points fixed by $\sigma^i\pi$. 
\end{proof}

Note that the dihedral group $D_n$ can be written as $D_n=\mathbb{Z}_n
\mathbb{Z}_2$, where $\mathbb{Z}_n\triangleleft D_n$ corresponds to
the subgroup generated by the rotations and $\mathbb{Z}_2$ is the
subgroup generated by the two reflections. More precisely,
$D_n=\mathbb{Z}_n \rtimes\mathbb{Z}_2$ when the non-identity element
of $\mathbb{Z}_2$ acts on $\mathbb{Z}_n$ by inversion. When $n$ is
odd, this is a particular case of a Frobenius group. A transitive
permutation group $G\leq \sym(n)$ is called a \textsl{Frobenius group}
if no non-trivial element fixes more than one point and some
non-trivial element fixes a point.  An alternative definition is as
follows: a group $G\leq \sym(n)$ is a Frobenius group if it has a
non-trivial proper subgroup $H$ with the condition that $H\cap
H^g=\{\id\}$, for all $g\in G\backslash H$, where $H^g=g^{-1}Hg$. This
subgroup is called a \textsl{Frobenius complement}.  Define the
\textsl{Frobenius kernel} $K$ of $G$ to be
\[
K=\left(G\backslash \bigcup_{g\in G} H^g\right)\cup\{\id\}.
\]
In fact, the non-identity elements of $K$ are all the derangement
elements of $G$. Note that if $G=KH\leq \sym(n)$ is a Frobenius group with kernel
$K$, then $|K|=n$ and $|H|$ must divide $n-1$; see \cite[Section
3.4]{dixon2012permutation} for proofs. This implies that the Frobenius
groups are relatively small transitive subgroups of $\sym(n)$. We also
observe the following.

\begin{lem}\label{size_of_G_x_frobenius}
  If $G=KH\leq \sym(n)$ is a Frobenius group with kernel $K$, then
  $|G_x|=|H|$, for any $x\in\{1,2,\dots,n\}$.
\end{lem}
\begin{proof}
Let $x\in\{1,2,\dots,n\}$. By the orbit-stabilizer theorem we have
\[
|x^G|=[G:G_x]=\frac{|G|}{|G_x|}=\frac{|K||H|}{|G_x|}=\frac{n|H|}{|G_x|},
\]
where $x^G$ is the orbit of $x$ under the action of $G$ on
$\{1,2,\dots,n\}$. Since this action is transitive, $|x^G|=n$.
\end{proof}

In order to find the maximum intersecting subsets of a Frobenius
group, we first describe their derangement graphs. Diaconis and
Shahshahani \cite{MR626813} show how to calculate the eigenvalues of
normal Cayley graphs using the irreducible representations of the
group (we refer the reader to \cite[Chapter 4]{BahmanAhmadi} for a
detailed proof). We will simply state the formula for the derangement graphs.

\begin{thm}\label{Diaconis} 
Let $G$ be a group. The eigenvalues of the graph $\Gamma_G$ are given by
\[
\eta_{\chi}=\frac{1}{\chi(\id)}\sum_{x\in \dd_G}\chi(x),
\]
where $\chi$ ranges over all irreducible characters of $G$. Moreover,
the multiplicity of $\eta_{\chi}$ is $\chi(\id)^2$.
\end{thm}

We will make use of the following classical result (see \cite[Theorem
18.7]{huppert1998character}).

\begin{thm}\label{irrs_of_frobenius}
  Let $G=KH$ be a Frobenius group with the kernel $K$. Then the
  irreducible representations of $G$ are of the following two types:
\begin{enumerate}[(a)]
\item Any irreducible representation $\Psi$ of $H$ gives an
  irreducible representation of $G$ using the quotient map $H\cong
  G/K$. These give the irreducible representation of $G$ with $K$ in
  their kernel.
\item If $\Xi$ is any non-trivial irreducible representation of $K$,
  then the corresponding induced representation of $G$ is also
  irreducible. These give the irreducible representations of $G$ with
  $K$ not in their kernel. \qed
\end{enumerate}
\end{thm}

The following fact will be used in the proof of the next theorem; the
reader may refer to \cite[Section 1.10]{MR1824028} for a proof.

\begin{prop}\label{sum_of_dim_of_irrs}
  Let $\{\Psi_i\}$ be the set of all irreducible representations of a
  group $G$. Then
\[
\sum_i(\dim \Psi_i)^2=|G|.\qed
\]
\end{prop}
Now we can describe the derangement graph of the Frobenius groups.
\begin{thm}\label{gamma_of_frobenius} Let $G=KH\leq \sym(n)$ be a
  Frobenius group with the kernel $K$. Then $\Gamma_G$ is the disjoint
  union of $|H|$ copies of the complete graph on $n$ vertices.
\end{thm}
\begin{proof}
According to Theorem~\ref{Diaconis}, the eigenvalues  of $\Gamma_G$ are given by 
\[
\eta_{\chi}=\frac{1}{\chi(\id)}\sum_{\sigma\in \dd_G} \chi(\sigma),
\]
where $\chi$ runs through the set of all irreducible characters of
$G$. First assume $\chi$ is the character of an irreducible
representation of $G$ of type (a) in
Theorem~\ref{irrs_of_frobenius}. Then we have
\[
\eta_{\chi}=\frac{1}{\chi(\id)}\sum_{\sigma\in \dd_G} \chi(\sigma)
          =\frac{1}{\chi(\id)}\sum_{\sigma\in \dd_G} \chi(\id)=|\dd_G|=|K|-1=n-1.
\]
According to Theorem~\ref{Diaconis} and
Proposition~\ref{sum_of_dim_of_irrs}, the multiplicity of
$\eta_{\chi}=n-1$ is
\[
\sum_{\Psi\in\irr(H)}(\dim \Psi)^2=|H|.
\]
Furthermore, assume $\Xi$ is an irreducible representation of $G$ of
type (b) in Theorem~\ref{irrs_of_frobenius}, whose character is $\xi$
and let $\chi$ be the character of the corresponding induced
representation of $G$.  If $\sigma\in G\backslash K$, then $\sigma\in
H^g$, for some $g\in G$. Thus $x^{-1}\sigma x\in H^{gx}$, for any
$x\in G$; hence $x^{-1}\sigma x \notin K$. By the definition of an
induced representation (see \cite[Section 3.3]{MR1153249}), this
implies that $\chi(\sigma)=0$. On the other hand, let $\chi_{\ID}$ be
the trivial character of $G$. Since $\chi\neq \chi_{\ID}$, the inner
product of $\chi$ and $\chi_{\ID}$ is zero (see \cite[Section
2.2]{MR1153249}). Since
\[
\left\langle \chi , \chi_{\ID}\right\rangle 
 = \frac{1}{|G|} \sum_{\sigma\in G} \chi(\sigma) \chi_{\ID}(\sigma^{-1})
 =\frac{1}{|G|} \sum_{\sigma\in G} \chi(\sigma)
 =\frac{1}{|G|}\left(\chi(\id)+\sum_{\sigma\in \dd_G} \chi(\sigma)\right),
\]
and the characters $\chi$ and $\chi_1$ are orthogonal, we have that
\[
\sum_{\sigma\in \dd_G} \chi(\sigma)=-\chi(\id).
\] 
This yields
\[
\eta_{\chi}=\frac{1}{\chi(\id)}\sum_{\sigma\in \dd_G} \chi(\sigma)=\frac{-\chi(\id)}{\chi(\id)}=-1.
\]
 We have, therefore, shown that 
\[
\Spec(\Gamma_G)=\left( {\begin{array}{cc}
 n-1 & -1 \\
 |H| & |H|(n-1)\\
 \end{array} } \right).
\]
Now the theorem follows from the fact that any graph with these
eigenvalues is the union of $|H|$ complete graphs.
\end{proof}

We point out that this proof may not be the easiest method to describe
the graph $\Gamma_G$; however it is a nice example of an application
of character theory to graph theory.   Now we establish
the EKR property for the Frobenius groups.

\begin{thm}\label{EKR_for_frobenius} 
  Let $G=KH$ be a Frobenius group with kernel $K$. Then $G$ has the
  EKR property. Furthermore, $G$ has the strict EKR property if and
  only if $|H|=2$.
\end{thm}
\begin{proof}
  Using Theorem~\ref{gamma_of_frobenius}, the independence number of
  $\Gamma_G$ is $|H|$. This along with
  Lemma~\ref{size_of_G_x_frobenius} shows that $G$ has the EKR
  property. For the second part of the theorem, first note that if
  $|H|=2$ and $S$ is an intersecting subset of $G$ of size two, then
  $S$ is, trivially, a point stabilizer.  To show the converse, we
  note that the cliques of $\Gamma_G$ are induced by the $|H|$ cosets
  of $K$ in $G$. Now suppose $|H|>2$ and let $S$ be a maximum
  intersecting subset of $G$ which is a coset of a point-stabilizer in
  $G$. Without loss of generality we may assume $S=\{\id=s_1,
  s_2,\dots, s_{|H|}\}$ and, hence, $S=G_x$, for some $x\in
  \{1,2,\dots,n\}$. Since $S$ is independent in $\Gamma_G$, no two elements of $S$
  are in the same coset of $K$ in $G$. Note that, the only fixed point
  of any non-identity element of $S$ is $x$. Let $s_3$ be in the coset
  $gK$. If all the elements of $gK$ fix $x$, then all the elements of
  $K$ will have fixed points, which is a contradiction. Hence there is
  an $s'_3\in gK$ which does not fix $x$. Now the maximum intersecting
  set $S'=\left(S\backslash \{s_3\}\right)\cup\{s'_3\}$ is not a
  point-stabilizer.
\end{proof}

Note that one can show the second part of
Theorem~\ref{EKR_for_frobenius} by a simple counting argument as
follows. There are $n^2$ cosets of point-stabilizers in $G$. Since
$\Gamma_G$ is the union of $|H|$ copies of the complete graph on $n$
vertices, the total number of maximum independent sets is
$n^{|H|}$. Therefore, in order for all the maximum independent sets of
$\Gamma_G$ to be cosets of point-stabilizers, the necessary and
sufficient condition is $|H|=2$.

Theorem~\ref{EKR_for_frobenius} provides an alternative proof for the
fact that the dihedral group $D_n$ has the strict EKR property, when
$n\geq 3$ is odd.

\section{The EKR property for external direct products}\label{EKR_products}

Given a sequence of permutation groups $G_1\leq \sym(n_1),\,\dots\,,
G_k\leq \sym(n_k)$, their \textsl{external direct product} is defined
to be the group $G_1\times \cdots\times G_k$, whose elements are
$(g_1,\ldots,g_k)$, where $g_i\in G_i$, for $1\leq i\leq k$, and the
binary operation is, simply, the component-wise multiplication. This
group has a natural action on the set $\Omega=[n_1]\times\cdots\times
[n_k]$ induced by the natural actions of $G_i$ on $[n_i]$; that is,
for any tuple $(x_1,\dots, x_k)\in \Omega$ and any element
$(g_1,\dots,g_k)\in G_1\times\cdots\times G_k$, we have
\[
(x_1,\dots, x_k)^{(g_1,\dots,g_k)}:=(x_1^{g_1},\dots, x_k^{g_k}).
\]
Let $G=G_1\times\cdots\times G_k$. Then the derangement graph
$\Gamma_G$ of $G$, is the graph with vertex set $G$ in which two
vertices $(g_1, \dots,g_k)$ and $(h_1,\dots,h_k)$ are adjacent if and
only if $g_ih_i^{-1}$ is a derangement, for some $1\leq i\leq k$. 

It is possible to express the derangement graph for a group that is an
external direct product at a product of derangement graphs. To do
this, we need two graph operations. The first is the
\textsl{complement}.  The complement of a graph $X$ is the graph
$\overline{X}$ on the same vertex set as $X$ and the edge set is
exactly all the pairs of vertices which are not adjacent in $X$.  The
second operation is the \textsl{direct product} of two graphs. Let $X$ and $Y$
be two graphs. The direct product of $X$ and $Y$ is the graph $X\times
Y$ whose vertex set is
\[
V(X\times Y)=V(X)\times V(Y),
\]
and in which two vertices $(x_1,y_1)$ and $(x_2,y_2)$ are adjacent if
$x_1\sim x_2$ in $X$ and $y_1\sim y_2$ in $Y$. We observe the following.

\begin{lem} 
Let the group $G=G_1\times\cdots\times G_k$ be the external direct product of the groups $G_1,\dots, G_k$, then
\[
\Gamma_G=\overline{\overline{\Gamma_{G_1}}\times\cdots\times\overline{\Gamma_{G_k}}}.
\]
\end{lem}
\begin{proof}
  By the definition of the external direct product, the vertices
  $(g_1, \dots,g_k)$ and $(h_1,\dots,h_k)$ of $\overline{\Gamma_G}$
  are adjacent if and only if $g_ih_i^{-1}$ has a fixed point, for any
  $1\leq i\leq k$. This is equivalent to the case where $g_i$ is
  adjacent to $h_i$ in $\overline{\Gamma_{G_i}}$, for any $1\leq i\leq
  k$. This occurs if and only if $(g_1, \dots,g_k)$ and
  $(h_1,\dots,h_k)$ are adjacent in
  $\overline{\Gamma_{G_1}}\times\cdots\times\overline{\Gamma_{G_k}}$. This
  completes the proof.
\end{proof}

In the next lemma, we evaluate the independence number of $\Gamma_G$.
If $G=G_1\times\cdots\times G_k$, define $p_i$ to be the projection of
$G$ onto the component $G_i$.

\begin{lem}\label{independence_of_external_prod} 
Let the group $G=G_1\times\cdots\times G_k$ be the external direct product of the groups $G_1,\dots, G_k$, then
\[
\alpha(\Gamma_G)=\alpha(\Gamma_{G_1})\times \cdots \times \alpha(\Gamma_{G_k}).
\]
\end{lem}
\begin{proof}
  Let $S_i$ be a maximum independent set in $\Gamma_{G_i}$, for any
  $1\leq i\leq k$. Then the set $S=S_1\times\cdots\times S_k$ is an
  independent set in $\Gamma_G$, proving that $\alpha(\Gamma_G)\geq
  \alpha(\Gamma_{G_1})\times \cdots \times \alpha(\Gamma_{G_k})$. On the
  other hand, let $S$ be a maximum independent set in $\Gamma_G$, then
  for any $1\leq i\leq k$, the set $p_i(S)$ is an independent set in
  $G_i$. Thus $|p_i(S)|\leq \alpha(\Gamma_{G_i})$. Since $S\subseteq
  p_1(S)\times\cdots\times p_k(S)$ the lemma follows.
\end{proof}

\begin{thm}\label{EKR_for_external_prod} 
  With the notation above, all the $G_i$ have the (strict) EKR
  property if and only if $G$ has the (strict) EKR property.
\end{thm}
\begin{proof}
Note that the stabilizer of any point $(x_1,\dots,x_k)\in \Omega$ in $G$ is 
\[
(G_1)_{x_1}\times\cdots\times (G_k)_{x_k}.
\]
We first prove the ``only if'' part of the theorem. If all the groups $G_i$ have the EKR property,
according to Lemma~\ref{independence_of_external_prod}, the maximum
size of an independent set in $\Gamma_G$ will be equal to
\[
|(G_1)_{x_1}|\times \cdots \times |(G_k)_{x_k}|,
\]
for some $(x_1,\dots,x_k)\in \Omega$; this proves that $G$ has the EKR
property. Furthermore, assume all the groups $G_i$ have strict EKR
property and let $S$ be a maximum independent set in $\Gamma_{G}$ that
contains the identity. This implies that for each $1\leq i\leq k$, the
set $p_i(S)$ is a maximum independent set in $\Gamma_{G_i}$ that
contains the identity; hence $p_i(S)=(G_i)_{x_i}$, for some $x_i\in
[n_i]$. Therefore, $S=G_{(x_1,\dots, x_k)}$, which shows that $G$ has
the strict EKR property.

To prove the ``if'' part, first assume $G$ has the EKR property. Then
the size of any independent set is bounded above by
$|G_{(x_1,\dots,x_k)}|$, for some $(x_1,\dots,x_k)\in
[n_1]\times\cdots\times[n_k]$.  Let $S_1$ be an independent set in
$G_1$. Then $S_1\times (G_2)_{x_2} \times \cdots\times (G_k)_{x_k}$ is
an independent set in $G$. Hence
\[
|S_1\times(G_2)_{x_2}\times\cdots\times(G_k)_{x_k}|\leq |G_{(x_1,\dots,x_k)}|=|(G_1)_{x_1}|\times\cdots\times |(G_k)_{x_k}|,
\]
thus $|S_1|\leq |(G_1)_{x_1}|$. This shows that $G_1$ and, similarly,
$G_2,\ldots,G_k$ have the EKR property.  Finally, assume $G$ has the
strict EKR property and let $S_1$ be a maximum independent set in
$G_1$ which contains the identity. This implies that $S_1\times
(G_2)_{x_2} \times \cdots\times (G_k)_{x_k}$ is a maximum independent
set in $G$; hence it must stabilize a point, say $(y_1,\dots,y_k)\in
[n_1]\times\cdots\times[n_k]$. Therefore, $S$ stabilizes $y_1$; that
is, $S\subseteq (G_1)_{y_1}$. Since $S$ has the maximum size, we
conclude that $S=(G_1)_{y_1}$, which completes that proof.
\end{proof}

We conclude this section by noting this theorem shows how to construct
infinite families of graphs that do not have (strict) EKR property.

\section{The EKR property for internal direct products}

The next group product that we consider is the internal direct product. Assume
$\Omega_1,\dots,\Omega_k$ are pair-wise disjoint non-empty subsets of
$\{1,2,\dots,n\}$, and consider the sequence $G_1\leq \sym(\Omega_1),\,\dots\,,
G_k\leq \sym(\Omega_k)$. Then their \textsl{internal direct product}
is defined to be the group $G_1\cdot G_2\cdot \cdots\cdot G_k$, whose
elements are $g_1g_2\cdots g_k$, where $g_i\in G_i$, for $1\leq i\leq
k$ and the binary operation is defined as follows: for the elements
$g_1g_2\cdot \cdots \cdot g_k$ and $h_1h_2\cdot \cdots \cdot h_k$ in
$G_1\leq \sym(\Omega_1),\,\dots\, G_k\leq \sym(\Omega_k)$,
\begin{equation}\label{binary_operation_internal_prod}
g_1g_2\cdots g_k\,\,\cdot \,\, h_1h_2\cdots h_k := (g_1h_1)(g_2h_2) \cdots (g_kh_k).
\end{equation}
Note that since the $\Omega_i$ do not intersect, any permutation in
$G_i$ commutes with any permutation in $G_j$, for any $1 \leq i\neq
j\leq k$; hence the multiplication in Equation
(\ref{binary_operation_internal_prod}) is well-defined.  This group
also has a natural action on the set $\Omega=\Omega_1\cup\cdots\cup
\Omega_k$ induced by the natural actions of $G_i$ on $\Omega_i$; that
is, for any $x\in \Omega$ and any element $g_1 g_2\cdots g_k\in
G_1\cdot G_2\cdot\cdots\cdot G_k$, we have
\[
x^{g_1 g_2\cdots g_k}:= x^{g_i},\quad \text{where } x\in \Omega_i.
\]
Let $G=G_1\cdot G_2\cdot\cdots\cdot G_k$. Then the derangement graph
of $G$ is the graph $\Gamma_G$ with vertex set $G$ in which two
vertices $g_1g_2\cdot \cdots \cdot g_k$ and $h_1h_2\cdot \cdots \cdot
h_k$ are adjacent if and only if $g_ih_i^{-1}$ is a derangement, for
all $1\leq i\leq k$.  In other words, $\Gamma_G$ is the direct product
of $\Gamma_{G_1},\dots,\Gamma_{G_k}$; that is
\begin{equation}\label{gamma_of_internal_prod}
\Gamma_G=\Gamma_{G_1}\times\cdots\times\Gamma_{G_k}.
\end{equation}

It is not difficult to see that
\[
\alpha(X\times Y)\geq\max\{\alpha(X)|Y|\,,\,\alpha(Y)|X|\}.
\]
This inequality can be strict for general graphs (see \cite{klavzar}),
but recently Zhang~\cite{Zhang2012832} proved that equality holds if both
graphs are vertex transitive.

\begin{thm} If $X$  and $Y$ are vertex-transitive graphs, then 
\[
\alpha(X\times Y)=\max\{\alpha(X)|Y|\,,\,\alpha(Y)|X|\}.\qed
\]
\end{thm}
The following can, then, be easily derived.
\begin{cor}\label{alpha_of_direct_prod_of_graphs}
If $X_1,\dots,X_k$ are vertex-transitive graphs, then 
\[
\alpha(X_1\times \cdots\times X_k)=\max_i\{\,\alpha(X_i)\prod_{\substack{j=1,\dots,n\\ j\neq i}}|V(X_j)|\,\}.\qed
\]
\end{cor}

\begin{thm}\label{EKR_for_internal_prod} 
  With the notation above, if all the groups $G_i$ have the EKR property,
  then $G$ also has the EKR property.
\end{thm}
\begin{proof}
  For any $x\in \Omega$, the stabilizer of $x$ in $G$ is
  $G_1\cdot\cdots\cdot G_{j-1}\cdot (G_j)_x\cdot G_{j+1}\cdot
  \cdots\cdot G_k$, where $x\in \Omega_j$. Hence
\[
|G_x|=|(G_j)_x|\prod_{\substack{i=1,\dots,n\\ i\neq j}}|G_i|.
\]
According to Corollary~\ref{alpha_of_direct_prod_of_graphs}
\[
\alpha(\Gamma_G)=\max_j\{\,\alpha(\Gamma_{G_j})\prod_{\substack{i=1,\dots,n\\ i\neq j}}|G_i|\,\}.
\]
Therefore, if all the
groups $G_i$ have the EKR property, then $G$ also has the EKR property.
\end{proof}

We point out that the converse of Theorem~\ref{EKR_for_internal_prod}
does not hold, examples of this are given in
Section~\ref{non-examples}.  Further, the ``strict''version of this
result does not hold.  In fact, Theorem~\ref{strict_EKR_for_Young}
gives examples of groups, which are the internal direct product of
groups that all have the strict EKR property, but do not satisfy the
strict EKR property.

Using Theorem~\ref{cyclic_groups}, one can observe the following.

\begin{cor} 
  For any sequence $r_1,\dots,r_k$ of positive integers, the internal
  direct product
  $\mathbb{Z}_{r_1}\cdot\mathbb{Z}_{r_2}\cdot\cdots\cdot
  \mathbb{Z}_{r_k}$ has the EKR property.\qed
\end{cor}

Let $\lambda=[\lambda_1,\dots,\lambda_k]$ be a partition of $n$; that
is, $\lambda_1\geq \lambda_2\geq\cdots\geq\lambda_k\geq 1$ and
$\sum_i\lambda_i=n$. Define a set partition of $\{1,2,\dots,n\}$ by
$\{1,2,\dots,n\}=\Omega_1\cup\cdots\cup\Omega_k$, where $\Omega_i=\{\lambda_1+
\cdots+\lambda_{i-1}+1,\dots,\lambda_1+\cdots+\lambda_i\}$. Then the
internal direct product $\sym(\Omega_1)\cdot
\sym(\Omega_2)\cdot\cdots\cdot \sym(\Omega_k)$ is called the
\textsl{Young subgroup} of $\sym(n)$ corresponding to $\lambda$ and is
denoted by $\sym(\lambda)$. An easy consequence of
Theorem~\ref{EKR_for_internal_prod} and Theorem~\ref{EKR_for_sym} is
the following.

\begin{cor} 
Any Young subgroup has the EKR property.\qed
\end{cor}

It is not difficult to see that $\Gamma_{\sym(n)}$ is connected if and
only if $n\neq 3$, the graph $\Gamma_{\sym(3)}$ is the disjoint union
of two complete graphs $K_3$ and $\Gamma_{\sym([2,2,2])}$ is
disconnected. From this we can deduce that if $\lambda=[3,2,\dots,2],
[3,3]$ or $[2,2,2]$, then $\Gamma_{\sym(\lambda)}$ will be
disconnected and one can find maximum independent sets which do not
correspond to cosets of point-stabilizers. More generally, if
$\lambda$ is any partition of $n$ which ``ends'' with one of these
three cases, then $\sym(\lambda)$ fails to have the strict EKR
property. Ku and Wong~\cite{KuW07} prove that these are the only Young
subgroups which don't have the strict EKR property. In other words,
they have proved the following.

\begin{thm}\label{strict_EKR_for_Young}
  Let $\lambda=[\lambda_1,\dots,\lambda_k]$ be a partition of $n$ with
  all parts larger than one. Then $\sym(\lambda)$ has the strict EKR
  property unless one of the following hold
\begin{enumerate}[(a)]
\item $\lambda_j=3$ and $\lambda_{j+1}=\cdots= \lambda_k= 2$, for some $1\leq j<k$;
\item $\lambda_k = \lambda_{k-1} = 3$;
\item $\lambda_k = \lambda_{k-1} = \lambda_{k-2} = 2$. \qed
\end{enumerate}
\end{thm}

\section{The EKR property for wreath products}

Let $G\leq \sym(m)$ and $H\leq \sym(n)$. Then the \textsl{wreath
  product} of $G$ and $H$, denoted by $G\wr H$ is the group whose set
of elements is
\[
(\underbrace{G\times \cdots\times G}_{n\,\,\text{times}})\times H,
\]
and the binary operation is defined as follows:
\[
(g_1,\dots,g_n,h)\cdot (g'_1,\dots,g'_n,h'):= (g_1 g'_{h(1)},\dots,g_n g'_{h(n)}\,,\,hh').
\]
In particular, note that the identity element of $G\wr H$ is
$(\id_G,\dots,\id_G,\id_H)$ and for any $(g_1,\dots,g_n,h)\in G\wr H$,
\[
(g_1,\dots,g_n,h)^{-1} = (g_{h^{-1}(1)}^{-1},\dots,g_{h^{-1}(n)}^{-1},h^{-1}).
\]
Note also that the size of $G\wr H$ is $|G|^n|H|$. 
We point out that  $G\wr H$ is in fact the ``semi-direct product''
\[
(\underbrace{G\times \cdots\times G}_{n\,\,\text{times}})\rtimes H,
\]
and the action of $H$ on $G\times \cdots\times G$ is
simply permuting the positions of copies of $G$ (see \cite[Section
2.5]{dixon2012permutation} for a more detailed discussion on
semi-direct products). It is not hard to see that this group is the
stabilizer of a partition of the set $\{1,2,\dots, nm\}$ into $n$ parts each of
size $m$.

Similar to what we have done in the previous sections, we will
describe the derangement graph $\Gamma_{G \wr H}$ as a subgraph of a
graph product of $\Gamma_G$ and $\Gamma_H$. In this case, we consider
the \textsl{lexicographic product} of graphs. Let $X$ and $Y$ be two
graphs. Then the \textsl{lexicographical product} $X[Y]$ is a graph
with vertex set $V(X)\times V(Y)$ in which two vertices
$(x_1,y_1),(x_2,y_2)$ are adjacent if and only if $x_1\sim x_2$ in $X$
or $x_1=x_2$ and $y_1\sim y_2$ in $Y$. An easy interpretation of
$X[Y]$ is as follows: to construct $X[Y]$, replace any vertex of $X$
with a copy of $Y$, and if two vertices $x_1$ and $x_2$ in $X$ are
adjacent, then in $X[Y]$ all the vertices which replaced the vertex
$x_1$ will be adjacent to all the vertices which replaced the vertex
$x_2$.

Note that if $S_X$ and $S_Y$ are independent sets in $X$ and $Y$,
respectively, then the $S_X[S_Y]$ is an independent set of vertices of
$X[Y]$. It is straight-forward to determine the size of the maximum
independent sets in a lexicographic product, see~\cite{Geller197587}
for details.

\begin{prop}\label{independence-of-lex} Let $X$ and $Y$ be graphs. Then 
\[
\alpha(X[Y])=\alpha(X)\alpha(Y).\qed
\]
\end{prop}

For any $x\in V(X)$, let $Y_x=\{x\}\times Y$ (this is the copy of $Y$
that replaces $x$ in $X$). If $S$ is a subset of the vertices of $X[Y]$,
define the projection of $S$ to $X$ as
\[
\proj_X(S)=\{x\in X\,\,|\,\, (x,y)\in S,\, \text{for some}\, y\in Y\}.
\]
Similarly, for any $x\in V(X)$ we define the projection of $S$ to $Y_x$ as
\[
\proj_{Y_x}(S)=\{y\in Y\,\,|\,\,(x,y) \in S\}.
\]
We can, then, observe the following.
\begin{prop}\label{max_indy_in_lex} 
  Let $X$ and $Y$ be graphs. If $S$ is an independent set in $X[Y]$ of
  size $\alpha(X)\alpha(Y)$, then $\proj_X(S)$ is a maximum
  independent set in $X$ and, for any $x\in V(X)$, the set $\proj_{Y_x}(S)$ is
  a maximum independent set in $Y_x$.\qed
\end{prop}

Assume $\Omega= \{1, \dots, m\} \times \{1,\dots ,n\}$, then the group $G\wr H$ acts on
$\Omega$ in the following fashion:
\begin{equation}\label{wreath_action_eq}
(x,j)^{(g_1,\dots,g_n,h)}:=(x^{g_j},j^h)=(g_j(x), h(j)),
\end{equation}
for any $(x,j)\in \Omega$ and $(g_1,\dots,g_n,h)\in G\wr H$. If
$(g_1,\dots,g_n,h)$ has a fixed point $(x,j)$, then $h(j)=j$ and
$g_j(x)=x$. Thus, it is not difficult to verify the following.
\begin{lem}\label{wreath_stabilizer}
For any pair $(x,j)\in \Omega$, the stabilizer of $(x,j)$ in $G\wr H$ is 
\[
\left(G\times\cdots \times\underset{j\text{th position}}{(G)_x}\times\cdots\times G\right)\times H_j.\qed
\]
\end{lem}

\begin{thm}\label{EKR_for_wreath}
If $G\leq \sym(m)$ and $H\leq \sym(n)$ have the EKR property, then $G\wr H$ also has the EKR property.
\end{thm}
\begin{proof}
  For convenience we let $W:=G\wr H$ and $P=G\times\cdots\times
  G$. Note that by the definition of the wreath product, $P$ is in
  fact the internal direct product of $G_1,\dots, G_n$, where
  $G_i\cong G$ and $G_i\leq \sym(\{1,2,\dots, m\} \times \{i\})$, for any $1\leq
  i\leq n$. Hence according to (\ref{gamma_of_internal_prod}), we have
\[
\Gamma_P=\underbrace{\Gamma_G\times \cdots\times \Gamma_G}_{n\,\,\text{times}}.
\]
Consider the lexicographic product $\Gamma=\Gamma_H[\Gamma_P]$. Define the map $f: \Gamma\to \Gamma_W$ by
\[
f(h,(g_1,\dots,g_n))=(g_1,\dots,g_n,h).
\]
We claim that $f$ is a homomorphism.  To prove this, assume
$(h,(g_1,\dots,g_n))$ and $(h',(g'_1,\dots,g'_n))$ are adjacent in
$\Gamma$. We should show that
\begin{equation}\label{f-is-hom}
(g'_1,\dots,g'_n,h')\cdot(g_1,\dots,g_n,h)^{-1}=
(g'_1  g_{h'h^{-1}(1)}^{-1},\dots,g'_n g_{h'h^{-1}(n)}^{-1},h'h^{-1})
\end{equation}
has no fixed point. By the definition of the lexicographic product,
either $h\sim h'$ in $\Gamma_H$ or $h=h'$ and $(g_1,\dots,g_n)\sim
(g'_1,\dots,g'_n)$ in $G$. In the first case, $h'h^{-1}$ has no fixed
point. Thus $(g'_1,\dots,g'_n,h')\cdot(g_1,\dots,g_n,h)^{-1}$ cannot
have a fixed point. In the latter case,
$(g'_1,\dots,g'_n)(g_1,\dots,g_n)^{-1}$ has no fixed point; thus,
according to (\ref{f-is-hom}),
\[
(g'_1,\dots,g'_n,h')\cdot(g_1,\dots,g_n,h)^{-1}= (g'_1  g_1^{-1},\dots,g'_n g_n^{-1},\id_H)
\]
cannot have a fixed point. Thus the claim is proved. 

We  can, therefore, apply the no-homomorphism lemma to get
\[
\frac{|V(\Gamma)|}{\alpha(\Gamma)}\leq \frac{|V(\Gamma_W)|}{\alpha(\Gamma_W)}.
\]
Therefore, using Proposition~\ref{independence-of-lex}, we have
\begin{equation}\label{alphas}
\alpha(\Gamma_W)\leq \alpha(\Gamma_P)\alpha(\Gamma_H).
\end{equation}
But since $G$ has the EKR property, according to
Theorem~\ref{EKR_for_internal_prod}, $P$ has the EKR property; this
means that there is a point $x\in \{1,2,\dots,m\}$ such that
\[
\alpha(\Gamma_P)=|P_x|.
\]
Similarly, since $H$ has the EKR property, there exists a $j\in\{1,2,\dots,n\}$ such that 
\[
\alpha(\Gamma_H)=|H_j|.
\]
This, along with Lemma~\ref{wreath_stabilizer}, implies that the
stabilizer of a point is the largest independent set in $\Gamma_W$.
\end{proof}

In the case of symmetric groups, we can say more.

\begin{prop} 
  The group $\sym(m)\wr \sym(n)$ has the EKR property. Furthermore, if
  $m\geq 4$, then $\sym(m)\wr \sym(n)$ has the strict EKR property.
\end{prop}
\begin{proof} 
  The first part follows from Theorem~\ref{EKR_for_wreath}. For the
  second part, as in the proof of Theorem~\ref{EKR_for_wreath}, we let
  $W=\sym(m)\wr \sym(n)$ and define the internal direct sum
\[
P=\sym([m]\times\{1\})\times\cdots\times \sym([m]\times\{n\}).
\]
Let $S$ be an intersecting subset of $W$ of maximum size, i.e. $S$ has
the size of a point-stabilizer in $W$. Without loss of generality we
assume that $S$ contains the identity element of $W$. Consider the
homomorphism $f:\Gamma_{\sym(n)}[\Gamma_P]\to \Gamma_W$ defined in the
proof of Theorem~\ref{EKR_for_wreath}. The function $f$ is an
injection; hence there is a copy of $\Gamma_{\sym(n)}[\Gamma_P]$ in
$\Gamma_W$. This implies that $S$ is an independent set in
$\Gamma_{\sym(n)}[\Gamma_P]$ of size
$\alpha(\Gamma_{\sym(n)})\alpha(\Gamma_P)$. 

For $x \in \sym(n)$, let $(\Gamma_P)_x$ denote the copy of $\Gamma_P$
that replaces $x$ from $\Gamma_{\sym(n)}$ in the graph
$\Gamma_{\sym(n)}[\Gamma_P]$. Then, according to
Proposition~\ref{max_indy_in_lex} and the fact that $P$ has the strict
EKR property (see Theorem~\ref{strict_EKR_for_Young}), we have that
the projection of $S$ to $(\Gamma_P)_x$ is a point-stabilizer in $P$;
denote this point stabilizer by $P_{(y_x,j_x)}$ where $y_x \in \{1,
\dots, m\}$ and $j_x \in \{1, \dots, n\}$.

Since $m\geq 4$ there is a permutation in $P_{(y_{x},j_{x})}$ that
fixes only the point $(y_{x}, j_{x})$.  This implies that if
$(y_{x_1}, j_{x_1})$ and $(y_{x_2}, j_{x_2})$ are different for two
vertices $x_1,x_2$ in the projection of $S$ to $\Gamma_{\sym(n)}$,
then the permutation in $P_{(y_{x_1},j_{x_1})}$ that fixes only the
point $(y_{x_1}, j_{x_1})$ and the permutation in
$P_{(y_{x_2},j_{x_2})}$ that fixes only the point $(y_{x_2}, j_{x_2})$
would not be intersecting.  Since $S$ is intersecting, the projection
of $S$ to each $(\Gamma_P)_x$ is the stabilizer of the point
$(y_{x},j_{x})$, which we will simply denote at $(y,j)$. Thus $S$ is
the stabilizer of $(y,j)$ in $\sym(m) \wr \sym(n)$.
\end{proof}

\section{Groups that do not have the EKR property}
\label{non-examples}

Our first example of a group that does not have the EKR property is
the Mathieu group $M_{20}$ (this group can also be thought of as the
stabilizer of a point in the Mathieu group $M_{21}$). The stabilizer
of a point in $M_{20}$ has size $48$ but there is an independent set
in $\Gamma_{M_{20}}$ of size $64$.

The group $M_{20}$ acts on the set $\{1, 2 \dots, 20\}$ and, under
this action, there is a system of imprimitivity that is comprised of
$5$ blocks, each of size $4$. Label these blocks by $B_i$. Consider
the following four possible ways a permutation can move these blocks:
\[
B_1 \rightarrow B_1, \quad B_2 \rightarrow B_2, \quad B_3 \rightarrow B_3, \quad B_4 \rightarrow B_5
\]
(three of the blocks are fixed and one is not).  Then the set of all
permutations that move the blocks in at least three of these four ways
forms an independent set of size $64$. There are many other sets of four
ways to move that blocks that will similarly produce independent sets
of size $64$. This shows that $M_{20}$ does not have the EKR
property. We conjecture that these are the largest independent sets in
$\Gamma_{M_{20}}$, and we leave it as an open problem whether or not
there are other ways to construct maximum independent sets in this graph.

The original EKR theorem for sets considered collections of sets in
which any two sets have at least $t$ elements in common for some
integer $t\geq 1$. Similarly, two permutations $\pi$ and $\sigma$ are
$t$-intersecting if $\pi\sigma^{-1}$ has at least $t$ fixed points. We
can then ask what is the size of the largest $t$-intersecting set of
permutations from a group? 

The largest set of $t$-intersecting permutations of a group $G \leq
\sym(n)$ is also the largest set of intersecting permutations in $G$
when we consider a different group action. Namely, if $G \leq
\sym(n)$, then $G$ has a natural action on the ordered $t$-sets from
$\{1,\dots,n\}$; so $G$ can also be considered as a subgroup of
$\sym(n \cdot (n-1) \cdot \cdots \cdot(n-t+1))$ with this action. Two
permutations are intersecting under this action if and only if they
are $t$-intersecting under the usual action on $\{1,\dots,n\}$.
Following the construction of $t$-intersecting sets of maximum size,
we can produce groups that do not have the EKR property.

For example, consider $4$-intersecting permutations in the group
$\sym(8)$ with the natural action on $\{1,2,\dots,8\}$. The stabilizer of $4$
points has size $4!=24$.  But the set of all permutations from
$\sym(8)$ that fix at least $5$ points from the set $\{1,2,3,4,5,6\}$
is also $4$-intersecting; but this set has size $26$. If we consider
$\sym(8)$ with its action on the ordered sets of size $4$ from $[8]$
then we have a subgroup of $\sym(1680)$ isomorphic to $\sym(8)$ that
does not have the EKR property.

This can be generalized for any $t$ to give more examples of groups
that do not have the EKR property.
\begin{lem}
  For $t\geq 4$ the size of a $t$-intersecting set in $\sym(2t)$ has
  size at least 
\[
(t^2+t-1)(t-2)!
\]
\end{lem}
\begin{proof} The set of all permutations that fix at least $t+1$ points from
$\{1,2,\dots, t+2\}$ is a $t$-intersecting set of size
\[
\binom{t+2}{t+2} (t-2)!  + \binom{t+2}{t+1} (t-1)(t-2)!\qed
\]
Note for all $t>0$ that $(t^2+t-1)(t-2)! \geq t!$ so these
are all examples of groups that do not have the EKR property.
\end{proof}

Finally, we note that it is possible to take the internal direct
product of a group that does not have the EKR property with one that
does and have the result be a group that does have the EKR
property. To see this, let $G_1 = M_{20}$ and $G_2=\sym(n)$ and let
$G$ be the internal direct product of $G_1$ and $G_2$. Then for $n$
sufficiently large we have that
\[
\alpha(\Gamma_G)=\max\{\alpha(\Gamma_{G_1})|G_2| , \alpha(\Gamma_{G_2})|G_1| \}=\alpha(\Gamma_{G_1})|(G_2)_{n}|=(n-1)!|G_1|,
\]
which is the size of a stabilizer of a point in $G$. This shows that
$G = G_1 \times G_2$ does have the EKR property while $G_1$ doesn't.

\section{Further Work}\label{future}

There is clearly much more work to be done to determine which groups
that have the EKR property. We would like to find more conditions on
groups to determine if they have either the EKR property or the strict
EKR property.

We would also like to investigate groups that do not have the either
the EKR property or the strict EKR property. In this paper we show that the
Frobenius groups do not have the strict EKR property, but this example
is not the most interesting since the derangement graph is the union
of complete of graphs. The Mathieu group $M_{20}$ is a more
intersecting example since it has non-trivial intersecting sets and
these set have an interesting structure. The next step is to find more
examples of groups that do not have the EKR property. The hope would
be to determine properties of the group that predicts when this
happens and to determine if the maximum independent sets have an
interesting structure.  For example, the groups $\PGL(3,q)$ never have
the strict EKR property since the stabilizer of a
hyperplane forms an intersecting set that has the same size as the
stabilizer of a point (in~\cite{MeagherS13} it is shown that these are
all the intersecting sets of maximum size).

Another direction to consider is based on the EKR theorem for
sets. The EKR theorem for $t$-intersecting $k$-subsets from $\{1,2,\dots,n\}$ has
been completely solved for all values of $n,k$ and
$t$~\cite{MR1429238}. Depending on the size of $n$, relative to
$k$ and $t$, the largest $t$-intersecting $k$-subsets are all the
subsets that contain $t+i$ elements from a set of $t+2i$ elements
(where $i$ is an integer that depends on $n,k$ and $t$).  Similarly,
for permutations, we can define the set of all permutations that fix
at least $t+i$ elements from a set of $t+2i$ elements. These sets are
natural candidates for intersecting sets of maximum size for groups in
which the EKR theorem does not hold. We plan to check these sets for
some small groups.

\end{document}